\documentclass[12pt,a4paper]{amsart}
\usepackage[top=35mm, bottom=35mm, left=30mm, right=30mm]{geometry}
\usepackage{mathptmx}
\usepackage{amsmath,amssymb,mathrsfs,bm}
\usepackage[colorlinks=true,citecolor=blue]{hyperref}
\usepackage{verbatim}

\theoremstyle{plain}
\newtheorem{thm}{Theorem}[section]
\newtheorem{lem}[thm]{Lemma}
\newtheorem{prop}[thm]{Proposition}
\newtheorem{cor}[thm]{Corollary}

\theoremstyle{definition}
\newtheorem{defn}[thm]{Definition}
\newtheorem{rem}[thm]{Remark}

\newcommand{\Z}{\mathbb{Z}}
\newcommand{\N}{\mathbb{N}}

\newcommand{\calK}{\mathcal{K}}
\newcommand{\calB}{\mathcal{B}}
\newcommand{\calC}{\mathcal{C}}

\def\cprime{$'$}
\renewcommand\MR[1]{\href{http://www.ams.org/mathscinet-getitem?mr=#1}{MR#1}}

\begin{document}
\title{Measure-theoretic sensitivity via finite partitions}

\author[J. Li]{Jian Li}
\date{\today}
\address[J.~Li]{Department of Mathematics, Shantou University, Shantou, Guangdong 515063, P.R. China
-- \and -- Guangdong Provincial Key Laboratory of Digital Signal and Image
Processing Techniques, Shantou University, Shantou, Guangdong 515063, P.R. China}
\email{lijian09@mail.ustc.edu.cn}

\begin{abstract}
For every positive integer $n\geq 2$,
we introduce the concept of measure-theoretic $n$-sensitivity
for measure-theoretic dynamical systems via finite measurable partitions,
and show that an ergodic system
is measure-theoretically $n$-sensitive but not $(n+1)$-sensitive
if and only if its maximal pattern entropy is $\log n$.
\end{abstract}
\maketitle

\section{Introduction}
The concept of \emph{sensitive dependence on initial conditions}
has attached lots of attention in recent years.
It captures the idea that a very small change in the initial condition can cause significantly different future behaviors.
Following the pioneer work by Guckenheimer  \cite{Guckenheimer1979},
Auslander and York \cite{Auslander1980} called that a continuous map $T\colon X\to X$ on the metric space $(X,d)$
is sensitive if there exists $\delta>0$ such that for any $x\in X$ and any $\varepsilon>0$
there is some $y\in X$ satisfying $d(x,y)<\varepsilon$ and $d(T^nx,T^ny)>\delta$
for some positive integer $n$.
See \cite{Glasner1993, Akin1996,Akin2003, Xiong2005, Shao2008, Ye2008, Li2013b, Li2015}
and references therein for the recent study of sensitivity in this line.

It is a natural question how to define some kind of sensitivity for measure-theoretic dynamical systems.
By a measure-theoretic dynamical system $(X,\calB,\mu,T)$,
we mean a Lebesgue space $(X,\calB,\mu)$ together with an invertible measure-preserving transformation
$T$ of $(X,\calB,\mu)$.
We shall also refer to $T$ or $(X,\mu,T)$ instead of $(X,\calB,\mu,T)$, for convenience.
There are serval attempts to study the sensitivity of measures via a suitable metric on the space
(see \cite{Abraham2002,Cadre2005,James2008,Zhang2008, Wu2009,Huang2011,Grigoriev2012,LiR2014,Garcia-Ramos2014}).
Recently, Morales \cite{Morales2013} studied measurable sensitivity via countable measurable partitions,
and Downarowicz and Lacroix  \cite{Downarowicz2014} studied
measure-theoretic chaos via refining sequence of finite measurable partitions.

In this paper, we propose a new approach to measure-theoretic sensitivity via finite measurable partitions.
Let $(X,\calB,\mu,T)$ be a measure-theoretic dynamical system and $n\geq 2$ be an integer.
We say that a finite measurable partition $\xi$ of $X$ is \emph{$n$-sensitive} if
there exists $\delta>0$ such that for every measurable set $A$ with positive measure,
there exist $n$ distinct points $x_1,x_2,\dotsc,x_n\in A$ and
a subset $F$ of $\Z_+$ with the lower density at least $\delta$
such that for every $k\in F$, $T^{k}x_1,T^{k}x_2,\dotsc,T^kx_n$ belong to different atoms of $\xi$.
The dynamical system $(X,\calB,\mu,T)$ is called
\emph{measure-theoretically $n$-sensitive} if
there exists an $n$-sensitive partition $\xi$ of $X$.
We will set up a relationship between the measure-theoretical sensitivity and maximal pattern entropy.

In 1958 Kolmogorov \cite{Kolmogorov1958} associated to
any measure-theoretic dynamical system  $(X,\calB,\mu,T)$
an isomorphism invariant, namely
the measure-theoretic entropy, $h_\mu(T)$.
Dynamical systems with positive entropy are random in certain sense,
and systems with zero entropy are said to be deterministic.
There are several ways to distinguish between deterministic systems.
One way to do this is to introduce the concept of sequence entropy.
In 1967 Kushnirenko \cite{Kusnirenko1967} studied the sequence entropy $h^\Gamma_\mu(T)$ of a
measure-theoretic dynamical system along an increasing sequence $\Gamma$ of non-negative integers, and
in particular he showed that an ergodic system is null,
that is $\sup_\Gamma h^\Gamma_\mu(T)=0$, if and only if it has a discrete spectrum.
Moreover, Pickel \cite{Pickel1969} and Walters (unpublished) showed that if $(X,\mu,T)$ is ergodic
then $\sup_\Gamma h^\Gamma_\mu(T)$ is $\infty$ or $\log n$ for some positive integer $n$.
Motivated by the maximal pattern complexity \cite{Kamae2002} studied by Kamae and Zamboni,
in 2009 Huang and Ye \cite{Huang2009} introduced the notion of maximal pattern entropy $h^*_\mu(T)$.
It is shown in \cite{Huang2009} that $h^*_\mu(T)=\sup_\Gamma h^\Gamma_\mu(T)$. 
Then for every ergodic system $(X,\mu,T)$, $h^*_\mu(T)$ is $\infty$ or $\log n$ for
some positive integer $n$.

The main result of this paper is  as follows.
\begin{thm} \label{thm:main-result}
If a measure-theoretic dynamical system $(X,\calB,\mu,T)$ is ergodic, then
\begin{enumerate}
  \item it is not measure-theoretically sensitive if and only if $h^*_\mu(T)=0$;
  \item it is  measure-theoretically $n$-sensitive but not $(n+1)$-sensitive
  for some integer $n\geq 2$ if and only if $h^*_\mu(T)=\log n$;
  \item it is measure-theoretically $n$-sensitive for every integer $n\geq 2$ if and only if $h^*_\mu(T)=\infty$.
\end{enumerate}
\end{thm}

The organization of the paper is as follows.
In Section 2, after introducing some basic notions,
we show some properties of measure-theoretic sensitivity.
In Section 3, we prove the main result Theorem~\ref{thm:main-result}.
In the final section, we show an alternative definition of measure-theoretic sensitivity,
that is for ergodic systems we can replace the lower density by upper density in the definition
of sensitive partition.

\section{Measure-theoretic sensitivity}

Denote by $\Z_+$ (resp. $\N$) the set of all non-negative integers (resp. the set of all positive integers).
For a subset $F$ of $\Z_+$, we define the \emph{upper density} $\overline{D}(F)$ of $F$ by
\[ \overline{D}(F)=\limsup_{N\to\infty} \frac{\#(F\cap[0,N-1])}{N},\]
where $\#(\cdot)$ is the number of elements of a set.
Similarly, the \emph{lower density} $\underline{D}(F)$ of $F$ is defined by
\[ \underline{D}(F)=\liminf_{N\to\infty} \frac{\#(F\cap[0,N-1])}{N}.\]
One may say $F$ that has \emph{density} $D(F)$ if $\overline{D}(F)=\underline{D}(F)$,
in which case $D(F)$ is equal to this common value.

By a \emph{measure-theoretic dynamical system}, we mean a quadruple $(X,\calB,\mu,T)$,
where $(X,\calB,\mu)$ is a Lebesgue space and
$T\colon X\to X$ is an invertible measure-preserving transformation.
A partition of $X$ is a collection of disjoint measurable subsets of $X$ whose union is $X$.
Note that we assume that atoms in a partition are measurable,
and we only consider finite partitions in most places.
Let $\xi$ and $\eta$ be two partitions.
We say that $\eta$ \emph{refines} $\xi$ or $\eta$ is a \emph{refinement} of $\xi$, denoted by $\xi\preccurlyeq \eta$,
if each atom of $\xi$ is a union of atoms in $\eta$.
The \emph{join} of two finite partitions  $\xi=\{P_1,P_2,\dotsc,P_k\}$ and $\eta=\{Q_1,Q_2,\dotsc,Q_m\}$,
denoted by $\xi\vee\eta$, is the partition into all sets of the form $P_i\cap Q_j$.
For simplicity, when the measure $\mu$ is clear, denote
\[\calB^+=\{B\in\calB\colon \mu(B)>0\}.\]

Now we can give the precise definition of measure-theoretic sensitivity.
\begin{defn}
Let $(X,\calB,\mu,T)$ be a measure-theoretic dynamical system  and $n\geq 2$ be an integer.
We say that a finite partition $\xi$ of $X$ is \emph{$n$-sensitive} if
there exists $\delta>0$ such that for every $A\in\calB^+$
there exist $n$ distinct points $x_1,x_2,\dotsc,x_n\in A$ and
a subset $F$ of $\Z_+$ with $\underline{D}(F)>\delta$
such that for every $k\in F$, $T^{k}x_1,T^{k}x_2,\dotsc,T^kx_n$ belong to different atoms of $\xi$.
The positive number $\delta$ is called an \emph{$n$-sensitive constant} for the partition $\xi$.

The dynamical system $(X,\calB,\mu,T)$ is called \emph{measure-theoretically $n$-sensitive} if
there exists an $n$-sensitive partition $\xi$ of $X$.
For the case $n=2$, we briefly say sensitivity instead of $2$-sensitivity.
\end{defn}


\begin{rem} Assume that a partition $\xi$ of $X$ is $n$-sensitive. Then we have the following facts:
\begin{enumerate}
  \item $\xi$ contains at least $n$ atoms;
  \item the measure $\mu$ is nonatomic;
  \item every  partition which refines $\xi$ is also $n$-sensitive;
  \item for every partition $\eta$, the partition $\eta\vee\xi$ refines $\xi$
  and is $n$-sensitive.
\end{enumerate}
\end{rem}

We first show that the notion of measure-theoretic sensitivity is invariant under isomorphisms.

\begin{prop}\label{prop:invariant}
If two measure-theoretic dynamical systems $(X,\calB,\mu,T)$ and $(Y,\calC,\nu,S)$ are isomorphic,
then $(X,\calB,\mu,T)$ is measure-theoretically $n$-sensitive  for some $n\geq 2$
if and only if so is $(Y,\calC,\mu,S)$.
\end{prop}
\begin{proof}
As $(X,\calB,\mu,T)$ and $(Y,\calC,\nu,S)$ are isomorphic,
there exist $X'\in\calB$ and $Y'\in\calC$ with $\mu(X')=\nu(Y')=1$ such that
$T(X')\subset X'$, $S(Y')\subset Y'$ and
there is an invertible measure-preserving transformation $\pi\colon X'\to Y'$
with $\pi(T(x))=S(\pi(x))$ for any $x\in X'$.

Suppose that $(X,\calB,\mu,T)$ is measure-theoretically $n$-sensitive.
Let $\xi$ be an $n$-sensitive partition of $X$ and
$\xi'$ be the restriction of $\xi$ to $X'$.
Let $\eta'$ be the partition of $Y'$ obtained as the image of $\pi$ of $\xi'$
and $\eta$ be the partition of $Y$ consisting of the elements
of $\eta'$ and the null set $Y_0=Y\setminus Y'$.
We want to show that $\eta$ is an $n$-sensitive partition of $Y$.

Let $\delta>0$ be an $n$-sensitive constant with respect to $\xi$.
Fix a measurable set $C\in\calC^+$. Let $A=\pi^{-1}(C)\cap X'$.
Then $\mu(A)>0$. By the definition of $n$-sensitivity,
There exist $n$ distinct points $x_1,x_2,\dotsc,x_n\in A$ and $F\subset\Z_+$
with $\underline{D}(F)\geq \delta$
such that for every $k\in F$, $T^{k}x_1,T^{k}x_2,\dotsc,T^kx_n$ belong to different atoms of $\xi$.
Note that the orbits of $x_1, x_2,\dotsc, x_n$  are contained in $X'$, then
$T^{k}x_1,T^{k}x_2,\dotsc,T^kx_n$ belong to different atoms of $\xi'$.
Let $y_1=\pi(x_1) , y_2=\pi(x_2),\dotsc,y_n=\pi(x_n)$.
Then $y_1,y_2,\dotsc,y_n\in C$ and
for every $k\in F$, $S^{k}y_1,S^{k}y_2,\dotsc,S^ky_n$ belong to different atoms of $\eta'$,
and then also to different atoms of $\eta$.
Thus $\eta$ is $n$-sensitive, which ends the proof.
\end{proof}

\begin{rem}
By Proposition~\ref{prop:invariant}, we know that if two partitions $\xi$ and $\eta$ of $X$
are equal (mod $0$) then $\xi$ is $n$-sensitive if and only if so is $\eta$.
\end{rem}

The following result reveals that if a finite partition is sensitive then
for almost every pair in the product space
the orbits of those two points belong to different atoms
of the partition along a time set with positive lower density.

\begin{prop}\label{prop:pairs-sensitvity}
Let $(X,\calB,\mu,T)$ be a measure-theoretic dynamical system
and $\xi$ be a finite partition of $X$.
Then  $\xi$ is a sensitive partition if and only if
there exists $\delta>0$ such that for $\mu\times\mu$-almost every pair $(x,y)\in X\times X$
there exists a subset $F$ of $\Z_+$ with $\underline{D}(F)>\delta$
such that for every $n\in F$, $T^{n}x$ and $T^{n}y$ belong to different atoms of $\xi$.
\end{prop}

\begin{proof}
The sufficiency is obvious. Now we prove the necessity.
Enumerate the partition $\xi$ as $\{P_1,P_2,\dotsc,P_r\}$.
For every $\tau>0$, let $W_\tau$ denote the collection of pairs $(x,y)\in X^2$ satisfying
that there exists
a subset $F$ of $\Z_+$ with $\underline{D}(F)\geq \tau$
such that for every $k\in F$, $T^{k}x ,T^{k}y$ belong to different atoms of $\xi$.
It is not hard to check that
\[W_\tau=\bigcap_{t=1}^\infty\bigcup_{M=1}^\infty\bigcap_{N=M}^\infty
\bigcup_{K\in\calK_\tau(t,N)}\bigcap_{k\in K}\bigcup_{1\leq i\neq j\leq r}T^{-k}P_i\times T^{-k}P_j,\]
where $\calK_\tau(t,N)=\{K\subset\{0,1,\dotsc,N-1\}\colon |K|\geq N(\tau-\tfrac{1}{t})\}$.
Then $W_\tau$ is measurable.
Let $\lambda$ be a sensitive constant for $\xi$ and put $\delta=\lambda/2$.
We want to show that $\mu\times \mu(W_\delta)=1$.

Fix a point $x\in X$. For every $k\in\Z_+$ there exists a unique atom,
named $P_{i_k}$, in $\xi$ containing $T^kx$.
Let $W_\delta(x)=\{y\in X\colon (x,y)\in W_\delta\}$.
We can express $W_\delta(x)$ as
\[W_\delta(x)=\bigcap_{t=1}^\infty\bigcup_{M=1}^\infty\bigcap_{N=M}^\infty
\bigcup_{K\in\calK_\tau(t,N)}\bigcap_{k\in K}T^{-k}\biggl(\bigcup_{i\neq i_k} P_i\biggr).\]
So $W_\delta(x)$ is also measurable.
It is sufficient to show that $W_\delta(x)\cap A\neq\emptyset$ for every $A\in \calB^+$.
Fix $A\in\calB^+$. There exist two distinct points $y_1,y_2\in A$ and
a subset $F$ of $\Z_+$ with $\underline{D}(F)\geq\lambda$
such that for every $k\in F$, $T^{k}y_1, T^k y_2$ belong to different atoms of $\xi$.
So for every $k\in F$, there exists $a_k\in\{1,2\}$ such that
$T^kx, T^k y_{a_k}$ belong to different atoms of $\xi$.
Let $F_1=\{k\in F\colon a_k=1\}$ and $F_2=\{k\in F\colon a_k=2\}$.
Clearly, $F=F_1\cup F_2$.
There exists $a\in\{1,2\}$ such that $\underline{D}(F_a)\geq\lambda/2=\delta$.
Then $y_a\in W_\delta(x)\cap A$, which ends the proof.
\end{proof}

\smallskip
\noindent\textbf{Remark.}
In last two lines of the proof of Proposition~\ref{prop:pairs-sensitvity},
we have  $\underline{D}(F)\geq\lambda$ and $F=F_1\cup F_2$,
then we only have there exists $a\in\{1,2\}$ such that $\overline{D}(F_a)\geq\lambda/2=\delta$.
To get the desired result, we need a few more words.
Applying the Birkhoff ergodic theorem to the indicator function
of the set $\bigcup_{1\leq i\neq j\leq r} P_i\times P_j$, we get that 
$\mu\times\mu$-a.e. point $(z_1,z_2)\in X^2$ satisfies that
the density of the set of non-negative integers $k$ 
such that  $T^{k}z_1, T^k  z_2$ belong to different atoms of $\xi$ 
exists.
Denote by the collection of those points  $Z\subset X^2$.
By the Fubini theorem, there exists $Z_0\subset X$ with $\mu(Z_0)=1$
and for every $x\in Z_0$, $\mu(Z(x))=\mu(\{y\in X\colon (x,y)\in Z\})=1$.
So in the the proof of Proposition~\ref{prop:pairs-sensitvity},
we should fix $x\in Z_0$ and fix $A\in \mathcal{B}^+$ with $A\subset Z(x)$.
Now we know $\overline{D}(F_a)\geq \delta$.
As $(x,y_a)\in Z$,  
we know that the density of the set of non-negative integers $k$ 
such that  $T^{k}x, T^k  y_a$ belong to different atoms of $\xi$ 
exists and then is at least $\delta$. 
Hence, $y_a\in W_\delta(x)\cap A$.
\smallskip

A measure-theoretic dynamical system $(X,\calB,\mu,T)$ is called \emph{ergodic} if
the only members $B$ of $\calB$ with $T^{-1}B=B$ satisfy $\mu(B)=0$ or $\mu(B)=1$.
For a point $x\in X$ and a measurable set $U$, put
\[N(x,U)=\{i\in\Z_+\colon T^ix\in U\}.\]
If $(X,\calB,\mu,T)$ is ergodic, by the well-known Birkhoff ergodic theorem
(e.g. see \cite{Walters1982}), for every $U\in\calB$ we have
the density of $N(x,U)$ exists and equals to $\mu(U)$ for $\mu$-a.e.\ $x\in X$.
For convenience, we will say that all those points satisfy
the assertion of the ergodic theorem with regard to $U$.

If the Lebesgue space $(X,\calB,\mu)$ has some ``good'' metric,
then we can use the metric instead of partitions to characterize
measure-theoretic sensitivity.
\begin{thm}\label{thm:metic-n-sensitive}
Let $(X,\calB,\mu,T)$ be an ergodic system.
Suppose that there exists a metric $d(\cdot,\cdot)$ on $X$ such that
 $\calB$ is the completion of the Borel $\sigma$-algebra of $(X,d)$ under $\mu$.
If the ergodic system $(X,\calB,\mu,T)$
is measure-theoretically $n$-sensitive for some $n\geq 2$,
then there exists $\varepsilon>0$ such that
for every  $A\in\calB^+$
there exist $n$ distinct points $x_1,x_2,\dotsc,x_n\in A$ such that
\[\liminf_{N\to\infty}\frac{1}{N}\sum_{k=0}^{N-1}\min_{1\leq i<j\leq n} d(T^kx_i,T^kx_j)>\varepsilon.\]
Furthermore, if $(X,d)$ is compact, the converse is also true.
\end{thm}
\begin{proof}
Let $\xi$ be an $n$-sensitive partition of $X$
and $\delta>0$ be an $n$-sensitive constant with respect to $\xi$.
By regularity of the measure, each atom $P$ of $\xi$ can be approximated in measure
by a sequence of its closed subsets  $(F_{P,m})_{m\geq 1}$ such that $\mu(P\setminus F_{P,m})<\frac{1}{m}\mu(P)$.
We let
\[U_{m}=X\setminus \bigcup_{P\in\xi_k} F_{P,m}.\]
Then $\mu(U_{m})<\frac{1}{m}$ for every $m\geq 1$.
Also let $s_{m}$ denote the positive minimal distance between points in different sets $F_{P,m}$,
$F_{P',m}$ with $P,P'\in\xi$.

Let $X_0$ be the collection of points in $X$ which, for at least one of sets $U_{m}$,
does not satisfy the assertion of the ergodic theorem. As $(X,\mu,T)$ is ergodic, $\mu(X_0)=0$.
For every  $A\in\calB^+$, there exist $n$ distinct points $x_1,x_2,\dotsc,x_n\in A\setminus X_0$
and $F\subset \Z_+$ with $\underline{D}(F)\geq\delta$
such that for every $k\in F$, $T^kx_1,T^kx_2,\dotsc,T^kx_n$ belongs to different atoms of $\xi$.

Choose a positive number $\lambda$ such that $\delta-n\lambda>0$
and a positive integer $m_0$ such that $\frac{1}{m_0}<\lambda$.
Then $\mu(U_{m_0})<\lambda$.
Let $F_i=N(x_i,U_{m_0})$ for $i=1,2,\dotsc,n$.
We have $D(F_i)=\mu(U_{m_0})<\lambda$ for $i=1,2,\dotsc,n$.
Let $F'=F\setminus (F_1\cup F_2\cup\dotsb\cup F_n)$.
Then
\[\underline{D}(F')\geq \underline{D}(F)-\sum_{i=1}^nD(F_i)\geq\delta-n\lambda>0\]
and for every $k\in F'$,
$T^kx_1,T^kx_2,\dotsc,T^kx_n$ belong to different sets $F_{P,m_0}$,
$F_{P',m_0}$ with $P,P'\in\xi$.
So for every $k\in F'$, $\min_{1\leq i<j\leq n}d(T^kx_i,T^kx_j)\geq s_{m_0}$.
Put $\varepsilon=\frac{1}{2}(\delta-n\lambda)s_{m_0}$. Then
\[\liminf_{N\to\infty}\frac{1}{N}\sum_{k=0}^{N-1}\min_{1\leq i<j\leq n}d(T^kx_i,T^kx_j)
\geq s_{m_0}\liminf_{N\to\infty}\frac{1}{N} \#(F'\cap[0,N-1])>\varepsilon.\]

Now assume that $(X,d)$ is compact. We first have the following observation: if
\[\liminf_{N\to\infty}\frac{1}{N}\sum_{k=0}^{N-1}\min_{1\leq i<j\leq n}d(T^kx_i,T^kx_j)>\varepsilon,\]
then there exists $\delta>0$ (depending only on $\varepsilon$ and the distance of $X$)
such that there exists a subset $F$ of $\Z_+$ with $\underline{D}(F)\geq\delta$ such that
$\min_{1\leq i<j\leq n}d(T^kx_i,T^kx_j)>\delta$.
Since $(X,d)$ is compact,
there exists a finite partition $\xi$ of $X$ such that the diameter of
the largest atom in $\xi$ is less than $\delta$.
Note that if the distance of two points is greater than $\delta$ then they belong to different
atoms of $\xi$.
By the above observation,  we have that $\xi$ is an $n$-sensitive partition.
\end{proof}

Using Proposition~\ref{prop:pairs-sensitvity} and
repeating the proof of Theorem~\ref{thm:metic-n-sensitive},
we have the following result.
\begin{thm}\label{thm:top-model-pairs-sensitivity}
Let $(X,\calB,\mu,T)$ be an ergodic system.
Suppose that there exists a metric $d(\cdot,\cdot)$ on $X$ such that $\calB$ is the completion of the Borel $\sigma$-algebra of $(X,d)$ under $\mu$.
If  the ergodic system $(X,\calB,\mu,T)$
is measure-theoretically sensitive,
then there exists $\varepsilon>0$ such that
for $\mu\times\mu$-almost every pair $(x,y)\in X\times X$,
\[\liminf_{N\to\infty}\frac{1}{N}\sum_{k=0}^{N-1} d(T^kx,T^ky)>\varepsilon.\]
\end{thm}

\begin{cor}\label{cor:discrete-spectrum-non-sensitive}
If an ergodic  system $(X,\calB,\mu,T)$ has a discrete spectrum,
then it is not measure-theoretically sensitive.
\end{cor}
\begin{proof}
By the Halmos-von Neumann theorem (see Theorem 3.6 in \cite{Walters1982}),
there exists a compact meterizable abelian group $G$ with normalised Haar measure $m$
and an ergodic rotation $R_a\colon G\to G$, $g\mapsto ag$ such that
$(X, \mu,T)$ and $(G,m,R_a)$ are isomorphic.
On the group $G$ there is a metric $\rho$ which is rotation invariant in the sense
that $\rho(gx,gy)=\rho(x,y)$ for any $g,x,y\in G$.
By Theorem \ref{thm:metic-n-sensitive}, $(G,m,R_a)$ is not measure-theoretically sensitive
and then so is $(X,\calB,\mu,T)$.
\end{proof}

By a topological dynamical system, we mean a pair $(X,T)$
where $X$ is a compact metric space and $T\colon X\to\ X$ is a homeomorphism.
The Borel $\sigma$-algebra of $X$ is denoted by $\calB(X)$.
A  probability measure $\mu$ on $(X,\calB(X)$ is called \emph{$T$-invariant}
if the quadruple $(X,\calB(X),\mu,T)$ forms a  measure-theoretic dynamical system.
By the well known Krylov-Bogolioubov theorem, any topological dynamical system
always admits some invariant measure.
A topological dynamical system $(X,T)$ is called \emph{uniquely ergodic} if
there is only one invariant measure.

Let $(X,\calB,\mu,T)$ be a measure-theoretic dynamical system.
We say that a topological dynamical system $(\widehat{X},\widehat{T})$ is
a \emph{topological model} for  $(X,\calB,\mu,T)$  if  there exists an invariant measure
$\widehat{\mu}$ of $(\widehat{X},\widehat{T})$
such that the systems $(X,\calB,\mu,T)$
and $(\widehat{X},\calB(\widehat{X}),\widehat{\mu},\widehat{T})$ are measure theoretically isomorphic.
The well-known Jewett-Krieger's theorem \cite{Krieger1972}
states that every ergodic system has a uniquely ergodic topological model.

Inspired by Theorem \ref{thm:metic-n-sensitive}, we can define the sensitivity of invariant measures for
topological dynamical systems.

\begin{defn}
Let $(X,T)$ be a topological dynamical system and $n\geq 2$.
We say that an ergodic invariant measure $\mu$ of $(X,T)$ is \emph{mean $n$-sensitive}
if there exists $\varepsilon>0$ such that
for every  $A\in\calB^+$
there exist $n$ distinct points $x_1,x_2,\dotsc,x_n\in A$ such that
\[\liminf_{N\to\infty}\frac{1}{N}\sum_{k=0}^{N-1}\min_{1\leq i<j\leq n} d(T^kx_i,T^kx_j)>\varepsilon.\]
\end{defn}

Following the proof of Theorem \ref{thm:metic-n-sensitive},
we have the following characterization of measure-theoretical sensitivity.
\begin{prop}
Let $(X,\calB,\mu,T)$ be an ergodic system and $n\geq 2$.
Then the following conditions are equivalent:
\begin{enumerate}
  \item $(X,\calB,\mu,T)$ is measure-theoretically $n$-sensitive;
  \item for every topological model $(\widehat{X},\widehat{\mu},\widehat{T})$
  of $(X,\calB,\mu,T)$, $\widehat{\mu}$ is mean $n$-sensitive;
  \item there exists a topological model $(\widehat{X},\widehat{\mu},\widehat{T})$
  of $(X,\calB,\mu,T)$ such that $\widehat{\mu}$ is mean $n$-sensitive.
\end{enumerate}
\end{prop}

\section{Maximal pattern entropy and  measure-theoretic sensitivity}

The main aim of this section is to proof the main result (Theorem~\ref{thm:main-result}) of this paper.
First we need some preparations.

Let $(X,\calB,\mu)$ be a probability space.
The entropy of a finite partition $\xi=\{P_1,P_2,\dotsc,P_k\}$ of $X$ is
\[H_\mu(\xi)= -\sum_{i=1}^k\mu(P_i)\log\mu(P_i),\]
where $0\log 0$ is defined to be $0$.
We need the following two elementary results, see Theorems 4.1 and 4.3 of \cite{Walters1982} for example.

\begin{lem}\label{lem:partition-k}
If $\xi$ is a partition with $k$ atoms and $k\geq 2$, then
\[H_\mu(\xi)\leq \log k,\]
with equality if and only if $\mu(P)=\frac{1}{k}$ for each atom $P$ of $\xi$.
\end{lem}

\begin{lem}\label{lem:join-partition}
Let $\xi$ and $\eta$ be two finite partitions of $(X,\calB,\mu)$. Then
$H_\mu(\xi\vee\eta)\leq H_\mu(\xi)+H_\mu(\eta)$.
\end{lem}

According to Lemma~\ref{lem:partition-k}, it is easy to verify the following result.
\begin{lem}\label{lem:partition-near-k}
Let $\xi$ be a partition with $k$ atoms and $k\geq 2$.
For every $\varepsilon>0$ there exists $\lambda=\lambda(k,\varepsilon)>0$
such that if $H_\mu(\xi)>\log k-\lambda$
then $|\mu(P)-\frac{1}{k}|<\varepsilon$ for each atom $P$ of $\xi$.
\end{lem}

Let $(X,\calB,\mu,T)$ be a measure-theoretic dynamical system.
For a finite partition $\xi$ of $X$,
we define the \emph{sequence entropy} of $T$ with respect to $\xi$ along an increasing sequence $\Gamma=\{t_i\}_{i=1}^\infty$
of non-negative integers  by
\[h^\Gamma_\mu(T,\xi)=\limsup_{N\to\infty}\frac{1}{N} H_\mu\biggl(\bigvee_{i=1}^N T^{-t_i}\xi\biggr).\]
The \emph{sequence entropy} of $T$ along the sequence $\Gamma$ is
\[h^\Gamma_\mu(T)=\sup_\xi h^\Gamma_\mu(T,\xi),\]
where supremum is taken over all finite  partitions.

For a finite partition $\xi$ of $X$ and $k\in\N$, we define
\[p^*_{\mu,\xi}(k)=\sup_{(t_1<t_2<\dotsb<t_k)\in \Z_+^k} H_\mu\biggl(\bigvee_{i=1}^k T^{-t_i}\xi\biggr).\]
We define the \emph{maximal pattern entropy} of $T$ with respect to $\xi$ by
\[h^*_\mu(T,\xi)=\lim_{k\to\infty}\frac{1}{k}p^*_{\mu,\xi}(k).\]
It is easy to see that $\{p^*_{\mu,\xi}(k)\}_{k=1}^\infty$ is a sub-additive sequence, that is
for every $u,v\in\N$
\[p^*_{\mu,\xi}(u+v)\leq p^*_{\mu,\xi}(u)+p^*_{\mu,\xi}(v).\]
Then we have
\[h^*_\mu(T,\xi) =\inf_{k\geq 1} \frac{1}{k}p^*_{\mu,\xi}(k).\]
The \emph{maximal pattern entropy} of $T$ is
\[h^*_\mu(T)=\sup_\xi h^*_\mu(T,\xi),\]
where supremum is taken over all finite partitions.

\begin{thm}[\cite{Huang2009}]
For each finite partition $\xi$ of $X$,
$h^*_\mu(T,\xi)=\sup_\Gamma h^\Gamma_\mu(T,\xi)$.
Moreover, $h^*_\mu(T)=\sup_\Gamma h^\Gamma_\mu(T)$.
\end{thm}

We will use the following result, which is implicit in~\cite{Gillis1936},
see also \cite[Proposition 5.8]{Huang2011} for this version.

\begin{lem}\label{lem:sets-in-PS}
Let $(X,\calB,\mu)$ be a probability space and $\{E_i\}_{i=1}^\infty$
be a sequence of measurable sets with $\mu(E_i)\geq a>0$ for some constant $a$ and any $i\in\N$.
Then for any $k\geq 1$ and $\varepsilon>0$ there is $N=N(a,k,\varepsilon)$
such that for any tuple $\{s_1<s_2<\dotsb<s_n\}$ with $n>N$ there exist
$1\leq t_1<t_2<\dotsb<t_k\leq n$ with
\[\mu(E_{s_{t_1}}\cap E_{s_{t_2}}\cap\dotsb\cap E_{s_{t_k}})\geq a^k-\varepsilon.\]
\end{lem}

\begin{prop}\label{prop:logn-sensitive}
Let $(X,\calB,\mu,T)$ be an ergodic system and $n\geq 2$.
If $\xi=\{P_1,\dotsc,P_n\}$ is a partition of $X$ with $n$ atoms and
the maximal pattern entropy of $T$ with respect to $\xi$ is $\log n$,
then $\xi$ is an $n$-sensitive partition.
\end{prop}
\begin{proof}
Since $\xi$ has $n$ atoms, for every $k\in\N$ and $0\leq t_1<t_2<\dotsb<t_k$,
\[H_\mu\biggl(\bigvee_{i=1}^k T^{-t_i}\xi\biggr)\leq
\sum_{i=1}^kH_\mu(T^{-t_i}\xi)=kH_\mu(\xi)\leq k\log n,\]
and
\[p^*_{\mu,\xi}(k)\leq k\log n.\]
The maximal pattern entropy of $T$ with respect to $\xi$ is $\log n$,
that is
\[\inf_{k\geq 1} \frac{1}{k}p^*_{\mu,\xi}(k)=\log n,\]
so $p^*_{\mu,\xi}(k)=k\log n$ for every $k\geq 1$.

Let $\lambda=\lambda(n^n, \frac{1}{2n^n})$ be the constant in Lemma \ref{lem:partition-near-k}.
Fix a measurable set $A\in \calB^+$ and let $N=N(\mu(A),n,\frac{1}{2}\mu(A)^n)$
be the constant in Lemma \ref{lem:sets-in-PS}.
Since $p^*_{\mu,\xi}(N)=N\log n$,  there exist $0\leq t_1<t_2<\dotsb<t_N$ such that
\[H_\mu\biggl(\bigvee_{i=1}^N T^{-t_{i}}\xi\biggr)>N\log n -\lambda.\]
By Lemma \ref{lem:sets-in-PS}, there exist $1\leq s_1<s_2<\dotsb<s_n\leq N$ such that
\[\mu(T^{-t_{s_1}}A\cap T^{-t_{s_2}}A\cap\dotsb\cap T^{-t_{s_n}}A)\geq \mu(A)^n- \frac{1}{2}\mu(A)^n>0.\]
By Lemma \ref{lem:join-partition}, we know that
\[H_\mu\biggl(\bigvee_{i=1}^N T^{-t_{i}}\xi\biggr)
\leq H_\mu\biggl(\bigvee_{j=1}^n T^{-t_{s_j}}\xi\biggr)+
H_\mu\biggl(\bigvee_{j\in H} T^{j}\xi\biggr),\]
where $H=\{t_1,t_2,\dotsc,t_N\}\setminus\{t_{s_1},t_{s_2},\dotsc,t_{s_n}\}$.
Clearly, $H$ has $N-n$ integers and the partition $\bigvee_{j\in H} T^{j}\xi$
has at most non-empty $n^{N-n}$ atoms.
By Lemma~\ref{lem:partition-k}, we have
\begin{align*}
H_\mu\biggl(\bigvee_{j=1}^n T^{-t_{s_j}}\xi\biggr)&\geq H_\mu\biggl(\bigvee_{i=1}^N T^{-t_{i}}\xi\biggr)-
H_\mu\biggl(\bigvee_{j\in H} T^{j}\xi\biggr)\\
 &> (N\log n-\lambda)-(N-n)\log n=n\log n-\lambda.
\end{align*}
Let
\[W=T^{-t_{s_1}}P_1\cap T^{-t_{s_2}}P_2\cap\dotsb\cap T^{-t_{s_n}}P_n.\]
By Lemma~\ref{lem:partition-near-k}, we know that $|\mu(W)-\frac{1}{n^n}|<\frac{1}{2n^n}$
and then $\mu(W)>\frac{1}{2n^n}$.
By the ergodic theorem,
there exists a point $z\in T^{-t_{s_1}}A\cap T^{-t_{s_2}}A\cap\dotsb\cap T^{-t_{s_r}}A$
such that the density of $N(z,W)=\mu(W)$.
Let $x_i=T^{t_{s_i}}z$  for $i=1,2,\dotsc,n$. Then $x_i\in A$ for $i=1,2,\dotsc,n$.
For every $k\in N(z,W)$, $T^kx_i=T^k(T^{t_{s_i}}z)=T^{t_{s_i}}(T^k z)\in T^{t_{s_i}}W\subset P_i$ for $i=1,2,\dotsc,n$,
which implies that $T^kx_1,\dotsc,T^kx_n$ belong different atoms of $\xi$.
Thus $\frac{1}{2n^n}$ is an $n$-sensitive constant for $\xi$.
\end{proof}

For a measure-theoretic dynamical system  $(X,\calB,\mu,T)$, put
\[\calK_\mu=\{A\in\calB\colon h^*_\mu(T,\{A,A^c\})=0.\]
It is a $T$-invariant $\sigma$-algebra of $\calB$.
We say that $\calK_\mu$ is the \emph{Kronecker $\sigma$-algebra} of $(X,\calB,\mu,T)$.
Every $T$-invariant $\sigma$-algebra of $\calB$ determines a factor of $T$ in a natural way.
Let  $\pi\colon (X,\calB,\mu,T) \to (Z,\calK_\mu,\nu,S)$ be the factor map
to the Kronecker factor.
If $T$ is ergodic, then by Rohlin's skew product theorem we may write $T$ as a skew product
$(z,m)\mapsto(S(z),T_z(m))$ on $(Z\times M,\nu\times\rho)$.
Moreover, $(M,\rho)$ consists of $n$ atoms of measure $1/n$ for some positive integer $n$,
or $(M,\rho)$ is continuous (see \cite{Abramov1962}).

One of the key ingredients of the proof of Theorem \ref{thm:main-result}
is the following result by Pickel \cite{Pickel1969} and Walters (unpublished).
\begin{thm}\label{thm:structure}
Let $(X,\calB,\mu,T)$ be an ergodic system,
$\pi\colon (X,\calB,\mu,T) \to (Z,\calK_\mu,\nu,S)$ be the factor map to the Kronecker factor
and write $T$ as a skew product
$(z,m)\mapsto(S(z),T_z(m))$ on $(Z\times M,\nu\times\rho)$. Then
\begin{enumerate}
  \item $h^*_\mu(T)=\log n$ for some positive integer $n$ if and only if
  $(M,\rho)$ consists of $n$ atoms of measure $1/n$;
  \item $h^*_\mu(T)=\infty$ if and only if $(M,\rho)$ is continuous.
\end{enumerate}
\end{thm}

We divide the proof of Theorem~\ref{thm:main-result} into two parts,
see Propositions \ref{prop:mpe-logn-sensitive} and \ref{prop:mpe-logn-not-sensitive} below.

\begin{prop}\label{prop:mpe-logn-sensitive}
Let $(X,\calB,\mu,T)$ be an ergodic system and $n\geq 2$.
If the maximal pattern entropy of $T$ is at least $\log n$,
then it is measure-theoretically $n$-sensitive.
\end{prop}
\begin{proof}
Let $\pi$, $(Z,\calK_\mu,\nu,S)$ and $(Z\times M,\nu\times\rho)$ as in Theorem~\ref{thm:structure}.
As $h^*_\mu(T)\geq\log n$, there are two cases, one is $h^*_\mu(T)=\log m$ for some $m\geq n$,
and the other $h^*_\mu(T)=\infty$.

If $h^*_\mu(T)=\log m$, we know that $(M,\rho)$ consists of $m$ atoms of measure $1/m$.
Let $\xi$ be the partition $\{P_1,\dotsc,P_m\}$ of singletons in $M$.
Then $\rho(P_i)=1/m$ for $i=1,2,\dotsc,m$.
Let $\xi'=\{Z\times P_i\colon i=1,\dotsc,m\}$ be a partition of $X$.
Then $\xi'$ is independent from the Kronecker $\sigma$-algebra $\calK_\mu$, and thus
$h^*_\mu(T,\xi')=H_\mu(\xi')=H_\rho(\xi)=\log m$ (see \cite[Lemma 3.4]{Huang2009}).
By Proposition \ref{prop:logn-sensitive}, $\xi'$ is $m$-sensitive
and then also $n$-sensitive.

If $h^*_\mu(T)=\infty$, we know that $(M,\rho)$ is continuous.
Since $(M,\rho)$ is a Lebesgue space,
there exists a partition $\eta=\{Q_1,\dotsc,Q_n\}$ of $M$ with $\rho(Q_i)=\frac{1}{n}$ for $i=1,\dotsc,n$.
Let $\eta'=\{Z\times Q_i\colon i=1,\dotsc,n\}$ be a partition of $X$.
$\eta'$ is also independent from the Kronecker $\sigma$-algebra $\calK_\mu$,
and thus $h^*_\mu(T,\eta')=\log n$.
By Proposition \ref{prop:logn-sensitive} again, $\eta'$ is $n$-sensitive.
\end{proof}

\begin{cor}
If a non-trivial measure-theoretic dynamical system $(X,\calB,\mu,T)$ is weakly mixing,
then for every $n\geq 2$
it is measure-theoretically $n$-sensitive.
\end{cor}
\begin{proof}
Since $(X,\calB,\mu,T)$ is weakly mixing,
the Kronecker $\sigma$-algebra is trivial.
Then $h^*_\mu(T,\xi) =H(\xi)$ for every partition $\xi$ \cite{Pickel1969}.
So its maximal pattern entropy is infinite, as $(X,\calB,\mu,T)$ is non-trivial.
Now the result follows from Proposition~\ref{prop:mpe-logn-sensitive}.
\end{proof}

\begin{cor}
If an ergodic system $(X,\calB,\mu,T)$ has positive entropy,
then for every $n\geq 2$
it is measure-theoretically $n$-sensitive.
\end{cor}
\begin{proof}
It follows from the result that if the entropy of an ergodic system is positive,
then its maximal pattern entropy is infinite \cite{Huang2009}.
Now the result follows from Proposition~\ref{prop:mpe-logn-sensitive}.
\end{proof}

%

The following result can be regarded  as the opposite to Proposition~\ref{prop:mpe-logn-sensitive}.
\begin{prop} \label{prop:mpe-logn-not-sensitive}
If an ergodic  system is $(X,\calB,\mu,T)$ is $n$-sensitive for some $n\geq 2$,
then  the maximal pattern entropy of $T$ is at least $\log n$.
\end{prop}
\begin{proof}
We first consider the case $n=2$.
If the maximal pattern entropy of $T$ is less than $\log 2$, then it is zero.
Then by Kusnirenko's result in~\cite{Kusnirenko1967} $(X,\calB,\mu,T)$ has a discrete spectrum.
By Corollary \ref{cor:discrete-spectrum-non-sensitive},
$(X,\calB,\mu,T)$ is not measure-theoretically sensitive.
This is a contradiction. So the maximal pattern entropy of $T$ is at least $\log 2$.

Now assume that $n>2$. By the first case we know that
the maximal pattern entropy of $T$ is at least $\log 2$.
Assume that the maximal pattern entropy of $T$ is  $\log m$ for some $2\leq m<n$.
Let $\pi$, $(Z,\calK_\mu,\nu,S)$ and $(Z\times M,\nu\times\rho)$ as in Theorem~\ref{thm:structure}.
Since $h^*_\mu(T)=\log m$, we know that $(M,\rho)$ consists of $m$ atoms of measure $1/m$.
Let $\xi$ be the partition $\{Q_1,\dotsc,Q_{m}\}$ of singletons in $M$,
that is with $\rho(Q_i)=1/m$.
Let $\xi'=\{Z\times Q_i\colon i=1,\dotsc,m\}$ be a partition of $X$.
As $(X,\mu,T)$ is $n$-sensitive,
there exists an $n$-sensitive partition $\xi''$ refining $\xi'$.
By the construction of $\xi'$, there exists a finite partition $\eta$ of $Z$ such that
the partition $\zeta=\{R\times Q\colon R\in\eta, Q\in\xi\}$ refining $\xi''$.
Moreover, $\zeta$ is also an $n$-sensitive partition of $X$.

We claim that $\eta$ is a sensitive partition of $(Z,\calK_\mu,\nu,S)$.
Fix $A\in\calK_\mu^+$. Then $A\times Q_1\in\calB^+$ and
there exist $n$ distinct points $x_1,\dotsc,x_{n}$ in $A\times Q_1$
and a subset $F$ of $\Z_+$ with $\underline{D}(F)>\delta$ such that
for every $k\in F$, $T^{k}x_1,\dotsc,T^k x_{n}$ belong to different atoms of $\zeta$.
Express $x_1=(z_1,q_1),\dotsc,x_{n}=(z_{n},q_1)$.
It is clear that $z_1,\dotsc,z_{n}$ are pairwise distinct and belong to $A$.
Note that $\xi$ only has $m$ atoms.
By the pigeonhole principle, for every $k\in F$, there exist $z_{k,1}$ and $z_{k,2}$ such that
$S^k z_{k,1}$ and $S^k z_{k,2}$ belong to different atoms of $\eta$.
So there exists two distinct points $z_{i_1}$ and $z_{i_2}$ in the list $\{z_1,\dotsc,z_{n}\}$
such that there exists a subset $F'$ of $F$ with $\underline{D}(F')>\delta/n^2$
satisfying for every $k\in F'$, $S^k z_{i_1}$ and $S^k z_{i_2}$ belong different atoms of $\eta$,
which implies that $\eta$ is a sensitive partition of $(Z,\nu,S)$.

As $(Z,\calK_\mu,\nu,S)$ is the Kronecker factor of $(X,\calB,\mu,T)$,
the maximal pattern entropy of $(Z,\nu,S)$ is zero.
Then by the proof of the first case, $(Z,\calK_\mu,\nu,S)$ is not measure-theoretically sensitive.
This is a contradiction. So the maximal pattern entropy of $T$ is at least $\log n$.
\end{proof}

Now it is clear that Theorem~\ref{thm:main-result}
follows from Propositions \ref{prop:mpe-logn-sensitive} and \ref{prop:mpe-logn-not-sensitive}.

\section{An alternative definition of measure-theoretic sensitivity}

In this section, we define the measure-theoretic weak sensitivity
using the upper density of subsets of $\Z_+$ instead of the lower density.
It turns out that it is equivalent to measure-theoretic sensitivity for ergodic systems.

\begin{defn}
Let $(X,\calB,\mu,T)$ be a  measure-theoretic dynamical system and $n\geq 2$.
We say that a finite partition $\xi$ of $X$ is \emph{weakly $n$-sensitive} if
there exists $\delta>0$ such that for every $A\in\calB^+$
there exist $n$ distinct points $x_1,x_2,\dotsc,x_n\in A$ and
a subset $F$ of $\N$ with $\overline{D}(F)\geq\delta$
such that for every $k\in F$, $T^{k}x_1,T^{k}x_2,\dotsc,T^kx_n$ belong to different atoms of $\xi$.

The  dynamical system $(X,\calB,\mu,T)$ is called \emph{measure-theoretically weakly $n$-sensitive} if
there exists a weakly $n$-sensitive partition $\xi$ of $X$.
\end{defn}

It is routine to check that all the results in the previous two sections also hold for
measure-theoretic weak $n$-sensitivity, just replacing
the lower density of the subset $F$ of $\Z_+$ by the upper density
and the limit infimum by limit supremum.
Especially, similar result of Theorem~\ref{thm:main-result}
also holds for measure-theoretic weak $n$-sensitivity.
So we get that those two definitions of measure-theoretic
sensitivity are equivalent for ergodic systems.
More specifically, we have the following result.

\begin{thm}\label{thm:weakly}
Let $(X,\calB,\mu,T)$ be an ergodic system and $n\geq 2$.
Then $(X,\calB,\mu,T)$ is measure-theoretically $n$-sensitive
if and only if it is measure-theoretically weakly $n$-sensitive.
\end{thm}

\begin{rem}
It should be noticed that the author in \cite{Garcia-Ramos2014}
also introduced the conception of measurable sensitivity.
By Theorems~\ref{thm:main-result} and \ref{thm:weakly}, and Theorem 41 in \cite{Garcia-Ramos2014},
for ergodic systems the measurable sensitivity in Definition 40 of~\cite{Garcia-Ramos2014}
is equivalent to measure-theoretical $n$-sensitivity for the case $n=2$.
\end{rem}

\section*{Acknowledgement}
The author was supported in part by NSF of China (grant numbers 11401362 and 11471125).
The author would like to thank Siming Tu and Ruifeng Zhang for the careful reading and helpful comments.
Part of this work was done during the visit of the Chinese University of Hong Kong
in the summers of 2015. The author is grateful to professor De-Jun Feng for his warm hospitality.
The authors would also like to thank the anonymous referees for
their helpful suggestions concerning this paper.

\bibliographystyle{amsplain}

\providecommand{\bysame}{\leavevmode\hbox to3em{\hrulefill}\thinspace}
\providecommand{\MR}{\relax\ifhmode\unskip\space\fi MR }
\providecommand{\MRhref}[2]{%
  \href{http://www.ams.org/mathscinet-getitem?mr=#1}{#2}
}
\providecommand{\href}[2]{#2}

\end{document}